\documentclass[12pt]{article}
\usepackage[colorlinks,citecolor=blue,urlcolor=blue]{hyperref}
\usepackage{amsmath,amssymb,amsfonts,amsthm,graphicx}
\usepackage[usenames,dvipsnames]{color}
\newcommand{\C}{{\mathbb C}}

\newcommand{\G}{{\mathcal G}}
\newcommand{\be}{\begin{equation}}
\newcommand{\ee}{\end{equation}} 

\newcommand{\old}[1]{}

\newcommand{\red}[1]{{\color{red} #1}}

\newcommand{\R}{{\mathbb R}}

\newcommand{\T}{{\mathbb T}}
\newcommand{\K}{{\mathbb K}}
\newcommand{\GL}{{\mathrm{GL}}}

\newcommand{\Gr}{\mathrm{Gr}}
\newcommand{\DET}{\mathrm{DET}}
\renewcommand{\Pr}{\text{Pr}}

\newtheorem{theorem}{Theorem}
\newtheorem{thm}[theorem]{Theorem}

\newtheorem{lemma}[theorem]{Lemma}
\newtheorem{prop}[theorem]{Proposition}
\newtheorem{corollary}[theorem]{Corollary}
\newtheorem{question}{Open question}

\title{Multideterminantal measures}

\author{Richard Kenyon\footnote{Department of Mathematics, Yale University, New Haven; richard.kenyon at yale.edu.}}

\begin{document}

\date{}
\maketitle
\abstract{
We define multideterminantal probability measures, a family of probability measures on $[k]^n$ where $[k]=\{1,2,\dots,k\}$,
generalizing determinantal measures (which correspond to the case $k=2$). We give examples coming from 
the positive Grassmannian, from the dimer model and from the spanning tree model. 

We characterize kernels of \emph{pure} $k$-determinantal measures as those arising from $k$-tuples of
Grassmannian elements whose maximal minors have certain sign restrictions. As a special case we construct all kernels of
pure determinantal measures via a pair of elements of $Gr_{n_1,n}$ having corresponding Pl\"ucker coordinates of the 
same signs.

We also define and completely characterize determinantal probability measures on the permutation group
$S_n$. 
}

\section{Multideterminantal measures}

A \emph{determinantal measure} is a probability measure $\mu$ on $\{0,1\}^n$ defined by an $n\times n$ matrix,
the \emph{kernel} $T$. Point probabilities for $\mu$ are determinants of matrices constructed from $T$, and in particular determinants of principal minors of $T$ are probabilities of ``index inclusion'' events: for any subset $S\subset[n]$, 
\be\label{pminors}\Pr(\{x_i=1~\forall i\in S\})=\det(T_S^S).\ee

Determinantal measures were introduced by Macchi in \cite{Macchi}. They occur naturally in a number of settings, such as the edge inclusion probabilities in random spanning trees \cite{BP}, in the planar bipartite dimer model \cite{Kenyon.local}, and even in integer addition \cite{BDF}. Continuous versions include
the fermionic gas \cite{Macchi}, random analytic functions \cite{PV} and random matrix ensembles such as GUE, CUE and the Ginibre
ensembles (see \cite{Soshnikov}). See \cite{Lyons, Borodin} for more background.
Despite their ubiquity, determinantal measures remain mysterious:
even classifying kernels of determinantal measures is an open problem.
  
We study here a generalization of determinantal measures to  
probability measures on $[k]^n$ where $[k]=\{1,2,\dots,k\}$ is a finite set. We refer to $[k]$ as the set of \emph{colors}.
We call these measures \emph{$k$-determinantal measures}. 
In this case instead of a single kernel we have $k$ matrices $A_1,\dots,A_k$ summing to the identity. 
Probabilities of 
individual events are given by determinants of matrices formed from the $A_i$.
Specifically, the probability of the single element $(i_1,\dots,i_n)\in [k]^n$ is 
\be\label{singleform}\Pr((i_1,\dots,i_n))= \det(A^1_{i_1},A^2_{i_2},\dots,A^n_{i_n}),\ee
the determinant of the $n\times n$ matrix whose $j$th column 
is the $j$th column of matrix $A_{i_j}$. (We use superscripts on matrices to denote columns of the matrix.)
Probabilities of marginals like ``coordinates $1$ and $3$ have, respectively, colors $4$ and $2$''
are also given by determinants, see Section \ref{marginals}.
The classical determinantal measure is the case $k=2$ (after translating indices $\{0,1\}\mapsto\{1,2\}$)
where the kernel is the matrix $A_2$, and $A_1=I-A_2$.  

We give naturally occurring examples of $k$-determinantal measures, one family arising from the positive Grassmannian $\Gr_{n,kn}$,
another arising from the dimer model on a bipartite planar graph, and a third arising from the uniform 
spanning tree model (on a general connected graph). See Figure \ref{rgbtriangle} for a random sample from a $3$-determinantal point measure on the vertices in a triangular grid, coming from a spanning tree measure (see Section \ref{STexample} for details.)

\begin{figure}[htbp]
\begin{center}
\includegraphics[width=3in]{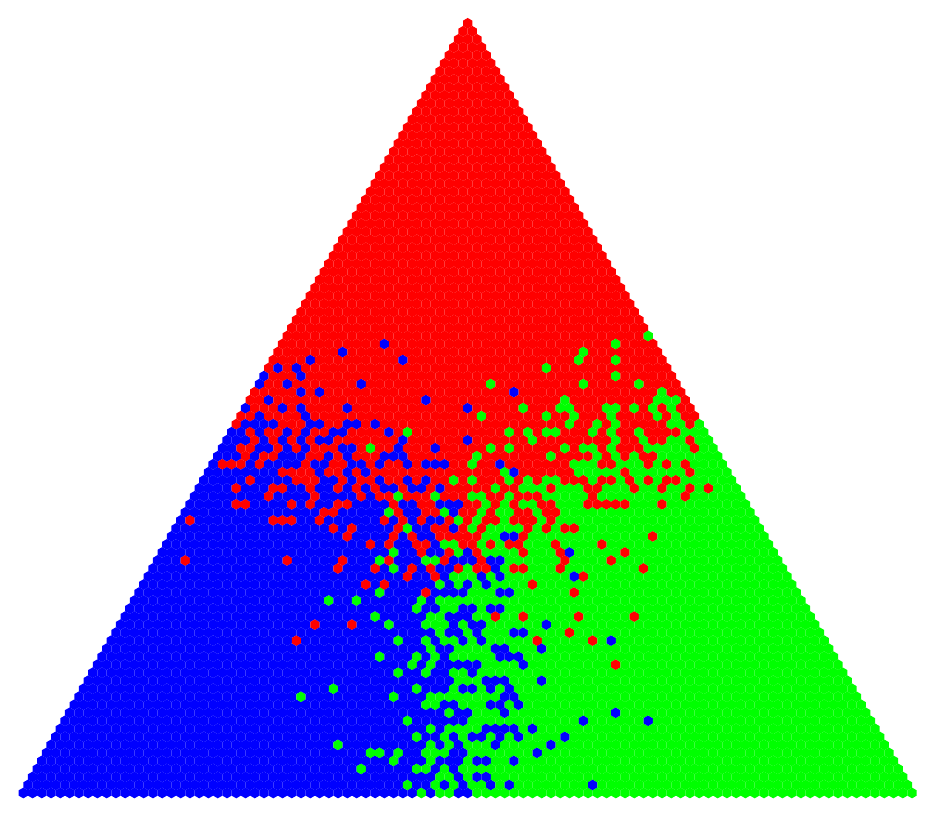}
\caption{\label{rgbtriangle}A $3$-determinantal measure (with $3$ colors, red, blue and green) on the vertices of a large triangular
region in the triangular grid. This example arises from a random spanning tree on the triangular grid with inhomogeneous conductances.}
\end{center}
\end{figure}

We also discuss \emph{symmetric} $k$-determinantal measures, which have the additional property that each 
matrix $A_i$ is symmetric.
For $k=2$ it is known that a symmetric matrix $A_1$ is the kernel of a determinantal measure if and only if all of its 
eigenvalues are in $[0,1]$, see \cite{Macchi}.
For $k\ge3$ we don't have an analogous characterization of matrices defining symmetric $k$-determinantal measures. However an analog of this eigenvalue property for $k\ge3$ relates to 
the \emph{characteristic polynomial} of the measure,
whose zero set is necessarily a \emph{Vinnikov curve} (for $k=3$) and a higher dimensional 
``Vinnikov variety'' for larger $k$, see Section \ref{vinnikovscn}.   

If the $A_i$ are symmetric and commute for $k\ge 3$, we have an analog of the ``sum of Bernoulli's" property 
for determinantal measures (Section \ref{commutingscn}). 

If the $A_i$ have ranks summing to $n$, we call $\mu$ a \emph{pure} $k$-determinantal measure.
In this case the measure is supported on points where the number of occurrences of each color $i$ is fixed (and equal to the rank of $A_i$).
We show how to encode such a measure via a single $n\times n$ matrix $L$, and we can give a concise description of point probabilities in terms of products of minors of $L$, see
Theorem \ref{minors}. This encoding allows us to characterize pure $k$-determinantal measures in terms of 
$k$ Grassmannian elements $L_i\in \Gr_{n_i,n}$ whose maximal minors have certain sign restrictions.
For example for $k=2$ a general pure $2$-determinantal measure is constructed from a pair of elements in the 
Grassmannian $\Gr_{n_1,n}$ having Pl\"ucker coordinates of the same signs. See Section \ref{mnevsection}.

When $k=n$ and each of the $A_i$ is of rank $1$, a $n$-determinantal measure is a probability measure on the group $S_n$ of permutations of $[n]$. We thus naturally construct 
\emph{determinantal random permutations}.
We give a complete classification of such measures, as those 
arising from Pfaffian bipartite graphs: see Theorem \ref{perm}.
Natural examples are given in Section \ref{permscn}.
\bigskip

\noindent{\bf  Acknowledgments.} 
We thank Omer Angel, Persi Diaconis and Nicholas Ovenhouse for discussions, and 
Steven Karp and Donovan Snyder for pointing out errors in an earlier version.
This research was supported by NSF grant DMS-1940932 and the Simons Foundation grant 327929.

\section{Basics}

We collect here a few basic facts about $k$-determinantal measures.

\subsection{Marginals}\label{marginals}

Note that from (\ref{singleform}), using the multilinearity of the determinant and the fact that the sum of the $A_i$ is $I$, the sum of all probabilities is $1$:
\begin{align}
\sum_{x\in[k]^n} \Pr(x) &= \sum_{x\in[k]^n}\det(A^1_{x_1},A^2_{x_2},\dots,A^n_{x_n})\nonumber\\
&=\label{wedge} 
(A_1^1+\dots+A_k^1)\wedge\dots\wedge(A_1^n+\dots+A_k^n)\\
&=e_1\wedge\dots\wedge e_n\nonumber\\
&=1.\nonumber
\end{align}

The marginal probabilities also have a simple form. For example, $Pr(x_i=j) = (A_j)_{ii}$. This follows by
restricting the $i$th term in the wedge product (\ref{wedge}) to be $A_j^i$ instead of the sum $A_1^i+\dots+A_k^i$.
Likewise restricting the $i_1$ term to be $A_{j_1}^{i_1}$ and the $i_2$ term to be $A_{j_2}^{i_2}$
gives
\be\label{marg}\Pr(x_{i_1}=j_1, x_{i_2}=j_2) = \det\begin{pmatrix}(A_{j_1})_{i_1,i_1}&(A_{j_2})_{i_1,i_2}\\(A_{j_1})_{i_2,i_1}&(A_{j_2})_{i_2,i_2}\end{pmatrix},\ee
and a similar expression holds for larger marginals. This illustrates the most useful property of $k$-determinantal measures:
their marginals can be quickly computed even when $n$ is large.

\old{
\subsection{Diagonal conjugation}

If matrices $A_1,\dots,A_k$ determine a $k$-determinantal measure on $[k]^n$, then for any diagonal matrix $S\in\GL_n(\R)$,
the matrices $SA_iS^{-1}$ determine the same measure.  This follows directly from (\ref{singleform}).
}

\subsection{$\GL_n(\R)_+$-Invariance}\label{invce}

If we have $k$ matrices $A_1,\dots,A_k$ for which the quantities in (\ref{singleform}) are all nonnegative
(and not all zero), but for which the sum $A_1+\dots+A_k$ is not necessarily the identity, we can construct a $k$-determinantal
measure by replacing each $A_i$
with $\tilde A_i:=MA_i$ where $M=(A_1+\dots+A_k)^{-1}$; these new matrices $\tilde A_i$ now sum to the identity
and the point probabilities are
\begin{align}\nonumber\Pr((i_1,\dots,i_n))&= \det((MA_{x_1})^1,(MA_{x_2})^2,\dots,(MA_{x_n})^n)\\
&\label{Mfactor}=\det(M(A_{x_1}^1),M(A_{x_2}^2),\dots,M(A_{x_n}^n))\\
&\nonumber= \det M\det(A^1_{x_1},A^2_{x_2},\dots,A^n_{x_n})
\end{align}
which are nonnegative since $\det M>0$. 

We call $A_1,\dots,A_k$ an \emph{unnormalized} $k$-determinantal measure.

As an example when $k=2$ we have the following.
\begin{prop}[\cite{Macchi}]\label{pm}
Matrices $A$ and $I$ form an
unnormalized $2$-determinantal measure
if and only if $A$ is a matrix with nonnegative principal minors.
\end{prop}

\begin{proof}
When $A_1=A$ and $A_2=I$, each determinant (\ref{singleform}) is a principal minor of $A$, and all principal minors appear.
\end{proof}

Thus when $A$ has nonnegative principal minors, the matrix $(I+A)^{-1}A$ is the kernel of a determinantal measure.
Moreover note that if $K$ is the kernel of a determinantal measure, then $K$ has nonnegative principal minors,
since these are marginals, by (\ref{pminors}). However not every kernel $K$ of a determinantal measure arises
from the above construction, that is such a $K$ is not necessarily of the form $(I+A)^{-1}A$ for a matrix $A$ with nonnegative
principal minors: \red{see e.g. \cite{Borodin}.}

\subsection{Characteristic polynomial}

The \emph{characteristic polynomial} $P(x_1,\dots,x_k)\in \R[x_1,\dots,x_k]$ of a $k$-determinantal measure 
is defined to be
$$P(x_1,\dots,x_k) = \det(x_1A_1+x_2A_2+\dots+x_k A_k).$$
It is homogeneous of degree $n$. 

\begin{lemma}We have
$$P(x_1,\dots,x_k) = \sum_{i_1+\dots+i_k=n} C_{i_1,\dots,i_k} x_1^{i_1}\dots x_k^{i_k}$$
where $C_{i_1,\dots,i_k}$ is the probability that, for each $j$, color $j$ occurs exactly $i_j$ times.
\end{lemma}

\begin{proof} This follows from (\ref{singleform}): the coefficient of $x_1^{i_1}\dots x_k^{i_k}$ in $P$ 
corresponds to summing over all determinants of matrices formed from $i_1$ of the columns of $A_1$, $i_2$ of the columns of $A_2$, etc., in some order. This is the probability that, for each $j$, index $j$ occurs $i_j$ times.
\end{proof}

 A $k$-determinantal measure $\mu$ induces a probability measure $\rho=\rho(\mu)$ on $\{0,1,2,\dots\}^k$ obtained by counting
 the number of occurrences of each color. In other words $\rho$ is the measure whose probability generating
 function is $P$: $\Pr_\rho(i_1,\dots,i_k) =  C_{i_1,\dots,i_k}.$

\begin{question}
What polynomials $P$ arise as characteristic polynomials of $k$-determinantal measures?
\end{question}

This is open even for $k=2$.

\subsection{Pure determinantal measures}

A $k$-determinantal measure on $[k]^n$ is \emph{pure} if the sum of the ranks of the $A_i$ is exactly $n$.
Then a point $(x_1,\dots,x_n)\in[k]^n$ has nonzero probability if and only if each index $i$ occurs exactly $n_i:=\text{rank}(A_i)$ times.
In particular the characteristic polynomial $P$ is a monomial: $P=\prod_{i=1}^k x_i^{n_i}$.
In this case moreover we have 
\begin{lemma}\label{pureproj}
For a pure $k$-determinantal measure each $A_i$ is a projection matrix.
\end{lemma}

\begin{proof}Define $B_i=\sum_{j\ne i}A_j$ so that $A_i+B_i=I$. Then $A_iB_i=B_iA_i$. But the images of $A_i$ and $B_i$ are complementary subspaces, so $A_iB_i=0$. Multiplying $A_i+B_i=I$ by $A_i$ on both sides we see that $A_i^2=A_i$. 
\end{proof}

Section \ref{projscn} contains more information about pure $k$-determinantal measures.

\subsection{Subdeterminantal measures}
Note that if $\mu$ is a $k$-determinantal measure and $S\subset[n]$ then $\mu$ restricted to $[k]^S$ is again $k$-determinantal,
with matrices $(A_i)_{S}^S$, the submatrices of $A_i$ with rows and columns indexed by $S$. This follows from the marginals property
generalizing (\ref{marg}).

\subsection{Forgetful maps}\label{forget}

Given a $k$-determinantal measure with matrices $A_1,\dots,A_k$, and a surjective map $\phi:[k]\to[\ell]$,
we can define an $\ell$-determinantal measure $B_1,\dots,B_\ell$ where $B_j=\sum_{i~:~\phi(i)=j}A_i$.
It is the image of $\mu$ under the map $[k]^n\to[l]^n$ induced by $\phi$. 
The fact that the image is a determinantal
measure follows from (\ref{singleform}) and the multilinearity. 

As an example, if $\ell=2$ and $\phi(j)=1$ for all $j$ except $j_0$, and $\phi(j_0)=2$,
then the image of $\mu$ is a determinantal measure.
This shows that each $A_j$,
and more generally the sum of any subset of $A_j$s, is the kernel of a usual determinantal measure.

\section{Natural examples}

\subsection{Grassmannian examples}

The Grassmannian $\Gr_{n,N}(\R)$ is the space of $n$-planes in $\R^{N}$. 
It can be presented as the space of $n\times N$ real matrices of rank $n$, modulo action on the left by $\GL_n(\R)$. 

The totally nonnegative Grassmannian $\Gr_{n,N}^{\ge}$ is the subset of $\Gr_{n,N}$ whose Pl\"ucker coordinates ($n\times n$ minors) 
are nonnegative.
Likewise define the totally positive Grassmannian $\Gr_{n,N}^{+}$ to be the subset of $\Gr_{n,N}$ whose Pl\"ucker coordinates are positive. 

Suppose $N=kn$ for some integer $k$. Given an element $G\in\Gr_{n,kn}^{\ge}$, represented as an $n\times kn$ matrix,
for $i\in[k]$ let $A_i$ be the $n\times n$ matrix formed from the columns $i,i+k,\dots,i+(n-1)k$ of $G$. 
The $k$-tuple $(A_1,\dots,A_k)$ forms an unnormalized $k$-determinantal measure (see Section \ref{invce}):
Let $M=\sum_{i=1}^k A_i$, and suppose $M$ is invertible. Define
$A'_i:=M^{-1}A_i$.
Note that the $A'_i$ are determined by and uniquely determine $G$, since $M^{-1}G=G$ as elements of  $\Gr_{n,kn}$. 
The $k$-tuple $(A'_1,\dots,A'_k)$ forms a $k$-determinantal measure.

We define
$\DET_{n,kn}\subset \Gr_{n,kn}$ to be the subset consisting of $k$-determinantal measures,
that is, the set of matrices $G\in  \Gr_{n,kn}$ such that forming the $A'_i$ as above results in a $k$-determinantal measure.
Note that $\Gr_{n,kn}^{\ge}\subset \DET_{n,kn}\subset \Gr_{n,kn}$.

Since elements of $\Gr_{n,kn}^{\ge}$ have a nice parameterization in terms of planar networks \cite{postnikov_06}, they provide
a tractable subclass of $k$-determinantal measures.

See Section \ref{projscn} for a different construction, of a pure $k$-determinantal measure from a set of $k$ Grassmannian elements
$L_i\in\Gr_{n_i,n}$ for $i=1,\dots,k$,
subject to certain sign restrictions.

\subsection{Dimer examples}

Let $\G=(B\cup W,E)$ be a bipartite planar graph having a dimer cover (a perfect matching). Let $\nu:E\to\R_{>0}$ be a positive edge weight function. 

Let $K$ be a Kasteleyn matrix for $\G$: this is a matrix with rows indexing the white vertices and columns indexing the black vertices,
with entries $$K(w,b) = \begin{cases}\pm \nu_e&w\sim b\\0&\text{else.}\end{cases}$$
where the signs are chosen according to the Kasteleyn rule \cite{Kast}: a face of length $\ell$ has $\frac{\ell}2+1\bmod2$ minus signs. 
Kasteleyn proved that in this setting $|\det K|$ is the weighted sum of dimer covers of $\G$, where the weight of a dimer cover
is the product of its edge weights. See \cite{Kenyon.lectures} for more information on dimers.

For edges $e=wb$ and $e'=w'b'$, let 
$$\K(e,e') = K(w,b)K^{-1}(b,w').$$
Note that the right-hand side does not depend on $b'$. 
By \cite{Kenyon.local}, $\K$ is the kernel of a determinantal measure $\tau$ on $\{0,1\}^E$. This is the inclusion
measure for dimers on edges of $\G$: for a point $(x_1,\dots,x_E)\in \{0,1\}^E$, the quantity
$\tau((x_1,\dots,x_E))$ is the probability that a random dimer cover of $\G$
covers exactly the edges $e$ for which $x_e=1$. 

Define for each $i\in[k]$ and each edge $e\in E$ a nonnegative edge weight $\nu_i(e)$, satisfying $\nu(e)=\sum_i \nu_i(e)$.
For a random dimer cover using edge $e$, color edge $e$ with color $i$ with probability proportional to $\nu_i(e)$,
independently for all edges, that is, with probability $p_i(e) = \frac{\nu_i(e)}{\nu_1(e)+\dots+\nu_k(e)}$.

Now consider the induced measure $\mu$ on $[k]^W$ which records the color of the dimer connected to $w$. 
We compute 
$$\Pr_\mu(\text{$w$ has color $i$}) = \sum_{b\sim w} p_i(bw)\Pr_\tau(bw) = \sum_{b\sim w} p_i(bw)K(w,b)K^{-1}(b,w).$$

Likewise for two vertices
$$\Pr_\mu(\text{$w_1$ has color $i_1$ and $w_2$ has color $i_2$})= 
\sum_{b_1\sim w_1}\sum_{b_2\sim w_2} p_{i_1}(b_1w_1)p_{i_2}(b_2w_2)\Pr_\tau(b_1w_1,b_2w_2)$$
where $b_1$ runs over neighbors of $w_1$ and $b_2$ runs over 
neighbors of $w_2$
\begin{align*}
& = \sum_{b_1}\sum_{b_2}p_{i_1}(b_1w_1)p_{i_2}(b_2w_2)\det\begin{pmatrix} K(w_1,b_1)K^{-1}(b_1,w_1)&K(w_2,b_2)K^{-1}(b_2,w_1)
\\K(w_1,b_1)K^{-1}(b_1,w_2)&K(w_2,b_2)K^{-1}(b_2,w_2)\end{pmatrix}\\
&=\det\begin{pmatrix} \sum_{b_1} p_{i_1}(b_1w_1)K(w_1,b_1)K^{-1}(b_1,w_1)&\sum_{b_2}p_{i_2}(b_2w_2)K(w_2,b_2)K^{-1}(b_2,w_1)\\
\sum_{b_1} p_{i_1}(b_1w_1)K(w_1,b_1)K^{-1}(b_1,w_2)&\sum_{b_2}p_{i_2}(b_2w_2)K(w_2,b_2)K^{-1}(b_2,w_2)\end{pmatrix}\\
&=\det\begin{pmatrix} (A_{i_1})_{w_1,w_1}&(A_{i_2})_{w_1,w_2}\\
(A_{i_1})_{w_2,w_1}&(A_{i_2})_{w_2,w_2}\end{pmatrix}
\end{align*}
where $A_i$ is the matrix with
$$(A_i)_{w_2,w_1}:=\sum_{b_1\sim w_1} p_i(b_1w_1)K(w_1,b_1)K^{-1}(b_1,w_2)=(K_iK^{-1})^t$$
where $K_i$ is the scaled Kasteleyn matrix with entries $K_i(w,b) = p_i(bw)K(w,b)$.

A similar argument works for all marginals, so the matrices $A_1,\dots, A_k$ form a $k$-determinantal measure on white vertices.
This proves:
\begin{thm}
The induced measure on $[k]^W$ is $k$-determinantal.
\end{thm}

For a simple example,  color the black vertices of $\G$ with arbitrary colors in $[k]$,
and assign edges adjacent to them with full weight of that color (other colors give that edge weight zero). For the  graph of Figure \ref{2by3} with edge weights $1$
\begin{figure}[htbp]
\begin{center}
\includegraphics[width=1.3in]{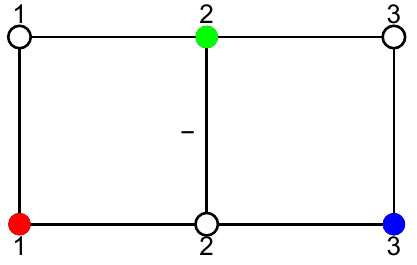}
\caption{\label{2by3}A small bipartite graph with Kasteleyn signs (edges without labels have sign $+$) and colored ``black'' vertices.}
\end{center}
\end{figure}
we have a Kasteleyn matrix
$$K=\begin{pmatrix}1&1&0\\1&-1&1\\0&1&1\end{pmatrix}.$$
With the indicated coloring of black vertices this leads to 
$$K_r=\begin{pmatrix}1&0&0\\1&0&0\\0&0&0\end{pmatrix} = K\begin{pmatrix}1&0&0\\0&0&0\\0&0&0\end{pmatrix}$$
and similarly for $K_g,K_b$, giving
\begin{align*}
A_r &= \left[K\begin{pmatrix}1&0&0\\0&0&0\\0&0&0\end{pmatrix}K^{-1}\right]^t = 
\begin{pmatrix}\frac23&\frac23&0\\\frac13&\frac13&0\\-\frac13&-\frac13&0\end{pmatrix}\\
A_g &= \left[K\begin{pmatrix}0&0&0\\0&1&0\\0&0&0\end{pmatrix}K^{-1}\right]^t= 
\begin{pmatrix}\frac13&-\frac13&\frac13\\-\frac13&\frac13&-\frac13\\\frac13&-\frac13&\frac13\end{pmatrix}\\
A_b &= \left[K\begin{pmatrix}0&0&0\\0&0&0\\0&0&1\end{pmatrix}K^{-1}\right]^t
=\begin{pmatrix}0&-\frac13&-\frac13\\0&\frac13&\frac13\\0&\frac23&\frac23\end{pmatrix}.
\end{align*}

The probability that white vertices $1,2,3$ are colored $r,b,g$ respectively is then
$$\Pr_\mu((r,b,g)) = \det(A_r^1,A_b^2,A_g^3) = \det\begin{pmatrix}\frac23&-\frac13&\frac13\\\frac13&\frac13&-\frac13\\-\frac13&\frac23&\frac13\end{pmatrix}=\frac13$$
and indeed, of the three dimer covers of $\G$, exactly one of them induces this coloring of the white vertices.

\subsection{Spanning tree examples}\label{STexample}

Let $\G=(V,E)$ be an arbitrary connected graph, with $V=\{v_0,v_1,\dots,v_{n-1}\}$ where $v_0$ is a distinguished root vertex.
Let $V'=V\setminus\{v_0\}$ the non-root vertices.
Let $c:E\to\R_{>0}$ be a nonnegative ``conductance'' function on edges. Let $\Delta:\R^V\to\R^V$ be the corresponding Laplacian:
$$(\Delta f)(v) = \sum_{u\sim v}c_{uv}(f(v)-f(u)).$$

We define the reduced Laplacian $\Delta':\R^{V'}\to\R^{V'}$ similarly:
$$(\Delta' f)(v) = \sum_{u\sim v}c_{uv}(f(v)-f(u))$$
where the sum is over $u\in V$, and $f(v_0)=0$ by definition.
In terms of matrices in the standard basis indexed by vertices, the matrix of the reduced 
Laplacian is obtained from the matrix of the Laplacian
by removing row and column $v_0$. The reduced Laplacian is invertible; its determinant is the weighted sum of spanning trees of $\G$,
where the weight of a tree is the product of its edge conductances \cite{Kirchhoff}.

For each edge $e\in E$ and each $i\in[k]$ define $c_i(e)\ge0$ such that $\sum_{i=1}^k c_i(e)=c(e)$.
Let $\Delta'_i$ be the Laplacian of $\G$, rooted at $v_0$, with conductances $c_i$.
Note that $\sum\Delta'_i=\Delta'$.
Let 
$$A_i:=(\Delta')^{-1/2}\Delta'_i(\Delta')^{-1/2}.$$ 
Then $\sum A_i=I$.

\begin{thm}
The matrices $A_i$ define a $k$-determinantal measure on $V\setminus\{v_0\}$.
\end{thm}  

\begin{proof}
Replace each edge of $\G$ with $k$ parallel edges with the same endpoints, one of each color, with the $i$th edge having conductance
$c_i(e)$. Let $\tilde\G$ be the new graph. There is a map from spanning trees of $\tilde\G$ to vertex colorings defined as follows.
Given a spanning tree $T$ of $\tilde\G$ rooted at $v_0$, color each vertex $v\ne v_0$ according to the color of the first edge on the unique path in $T$ from $v$ to the root. 

This mapping defines a measure $\mu$ on $[k]^{V'}$; it is the image of the weighted spanning tree measure on $\tilde\G$. 
We claim that $\mu$ is $k$-determinantal, with matrices $A_i$. 

For any coloring $\sigma\in [k]^{V'}$, the Directed Matrix Tree Theorem (see e.g. \cite{Chaiken}) says that the determinant
$$\det(\Delta_{\sigma(1)}^1,\dots,\Delta_{\sigma(n-1)}^{n-1})$$ 
of the matrix formed from the corresponding columns of the $\Delta_i$, is the weighted sum of spanning trees of $\tilde\G$ 
in which for each vertex $v$, the edge out of vertex $v$ (and in the direction of the unique path to $v_0$) has 
color $\sigma(v)$. 

The sum over all trees is $\det\Delta$, so 
$$\det(A_{\sigma(1)}^1,\dots,A_{\sigma(n-1)}^{n-1})$$
is the probability of coloring $\sigma$, as desired.
\end{proof}

An example is shown in Figure \ref{rgbtriangle} (for the point process) and Figure \ref{rgbtriangletree} (for the spanning tree).
We took a large triangle in the triangular grid, centered at the origin, with wired boundary conditions, and conductances which depend on position and orientation of edges: for a fixed parameter $q>0$, 
horizontal edges at coordinate $(x,y)$ have conductance $q^y$,
edges of slope $\sqrt{3}$ have conductances $q^{\frac{x\sqrt{3}}2-\frac{y}2}$, and 
edges of slope $-\sqrt{3}$ have conductances $q^{-\frac{x\sqrt{3}}2-\frac{y}2}$. Such a spanning tree can be sampled quickly using Wilson's algorithm \cite{Wilson}. 
Each vertex is then colored according to the direction of its outgoing edge (edge in the direction of the path to the wired outer boundary).
\begin{figure}[htbp]
\begin{center}
\includegraphics[width=4.5in]{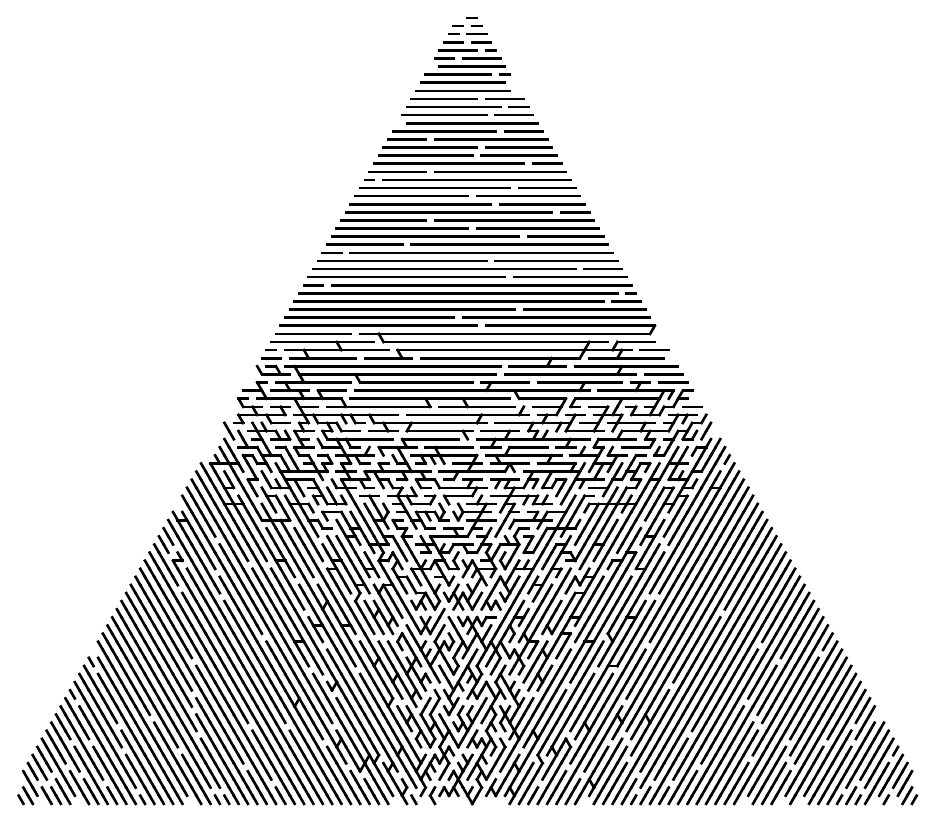}
\caption{\label{rgbtriangletree}A spanning tree with wired boundary and nonuniform conductances.}
\end{center}
\end{figure}

\section{Symmetric $k$-determinantal measures}

A $k$-determinantal measure is \emph{symmetric} if each $A_i$ is a symmetric matrix. 

\old{
As an example for $k=2$, let $M$ be a symmetric $n\times n$ matrix with nonnegative principal minors $M_S^S$ for any $S\subset[n]$.
Then $M$ and $I$ form a symmetric, unnormalized $2$-determinantal measure.
That is, letting $A=(I+M)^{-1}M,$ we have that
$A$ and $I-A$ define a symmetric $2$-determinantal measure.
To illustrate, for example when $n=5$
$$\Pr(1,0,0,1,1) = \frac{\det(M_1,I_2,I_3,M_4,M_5)}{\det(I+M)} = \frac{M_{\{1,4,5\}}^{\{1,4,5\}}}{\det(I+M)}$$
and likewise for other point probabilities.}

If $B_1,\dots,B_k$ are any real symmetric matrices, then for $\lambda>0$ large enough,
$A_i:=B_i+\lambda I$ will form an unnormalized symmetric $k$-determinantal measure.

Another example is provided by the spanning tree example of Section \ref{STexample}.

\old{
As a final example let $\G$ be a balanced planar bipartite graph with black vertices colored by $[k]$. Let a group $H$ act by color-preserving 
automorphisms of $\G$ and let $X$ be an orbit of white vertices. Let $k=|X|$. 
Then our construction above in the dimer section gives a symmetric $k$-determinantal measure on $X$. 
}

\subsection{Characteristic polynomial}\label{vinnikovscn}

When the $A_i$ are symmetric, the characteristic polynomial $P$ has some additional properties.
For $\vec u\in\R^{k-1}$, and $t\in\R$, along any line $t\mapsto (1,t u_2,\dots,t u_{k})$, the polynomial $P=P(t)$ as a function of $t$ has all real roots.
This follows from the definition of $P$, since (letting $B=u_2A_2+\dots+u_kA_k$) the one-variable polynomial
$P(t) = \det (A_1+t B) = \det B\det(B^{-1}A_1+tI)$ is a multiple of the characteristic polynomial of the symmetric matrix $B^{-1}A_1$,
which has real eigenvalues. Moreover since $A_1,B\ge0$ the roots of $P(t)$ are negative.

In particular in the case $k=2$, $P(x_1,x_2)$ factors as 
$$P=c\prod_{i=1}^k(x_1+a_ix_2)$$ for constants $c,a_i\ge 0$.

In the case $k=3$, $P$ defines a Vinnikov curve \cite{Vinnikov}, which is by definition the zero set of a polynomial $P(x,y) = \det (A+Bx+Cy)$
where $A,B,C$ are symmetric and positive (semi)definite. A Vinnikov curve is a ``generalized hyperbola'', see Figure \ref{vinnikov}.
\begin{figure}[htbp]
\begin{center}
\includegraphics[width=3in]{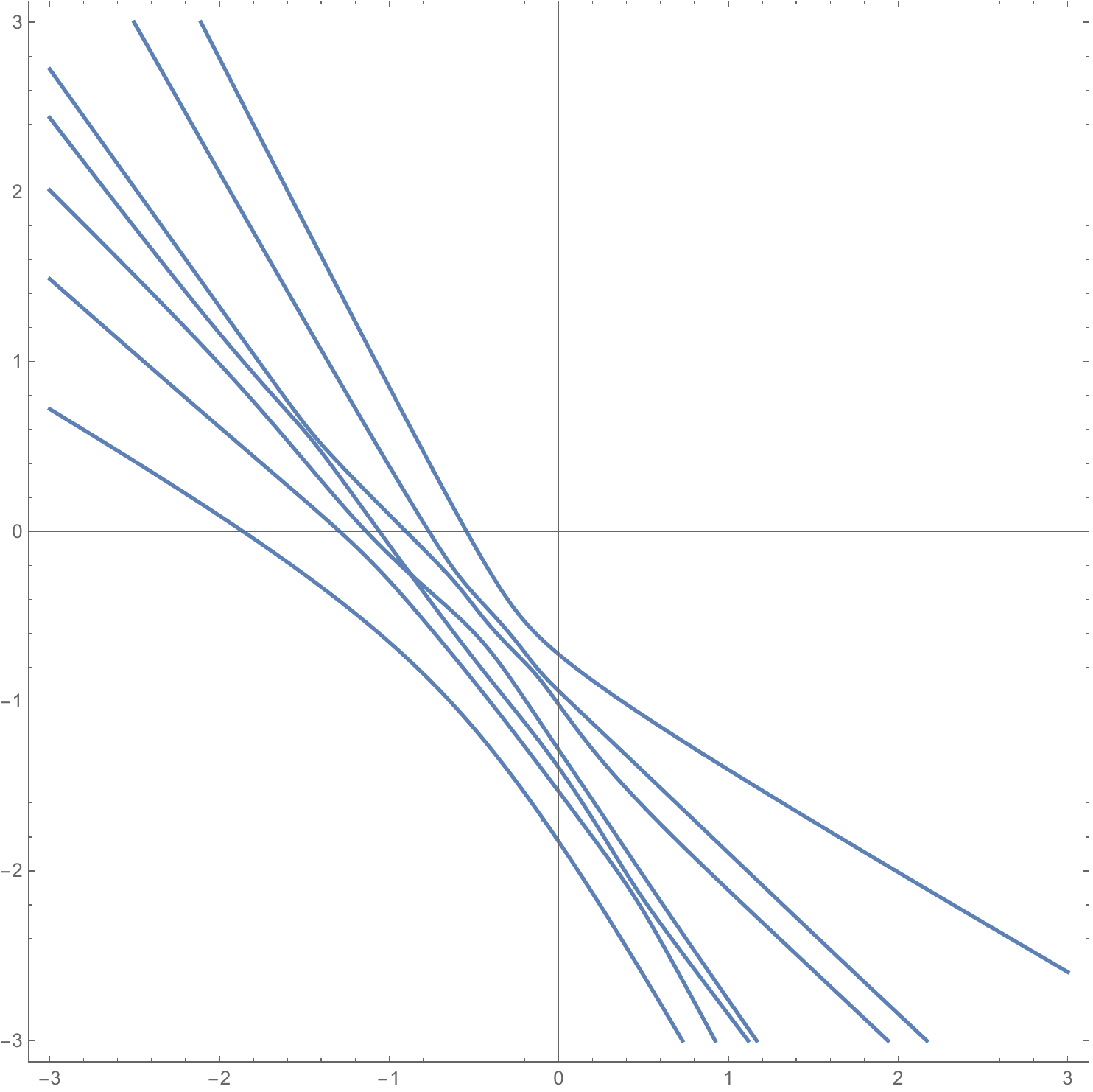}
\caption{\label{vinnikov}A Vinnikov curve of degree $n=7$ arising from the spanning tree $3$-determinantal measure on $K_8$ with randomly chosen conductances.}
\end{center}
\end{figure}
It has generically $n$ real components (which may touch for nongeneric parameter values),
and $n$ linear asymptotes whose slopes are the eigenvalues of $-C^{-1}B$; it intersects the $x$-axis at eigenvalues of $-C^{-1}A$
and the $y$-axis at eigenvalues of $-B^{-1}A$.

For larger $k$, we define a \emph{Vinnikov variety} to be the zero set of an expression
$$Q(x_1,\dots,x_k) = \det (x_1A_1+x_2A_2+\dots,x_kA_k)$$
where the $A_i$ are symmetric and positive semidefinite.
The zero set of the characteristic polynomial of a symmetric $k$-determinantal measure is a Vinnikov variety.

If $S\subset[n]$ is of size $|S|=n-1$ then by the interlacing property for eigenvalues of symmetric matrices,
the polynomial $P_S$ for the subdeterminantal measure 
has roots interlaced with those of $P$, as illustrated in figure \ref{interlaced}. 
\begin{figure}[htbp]
\begin{center}
\includegraphics[width=3in]{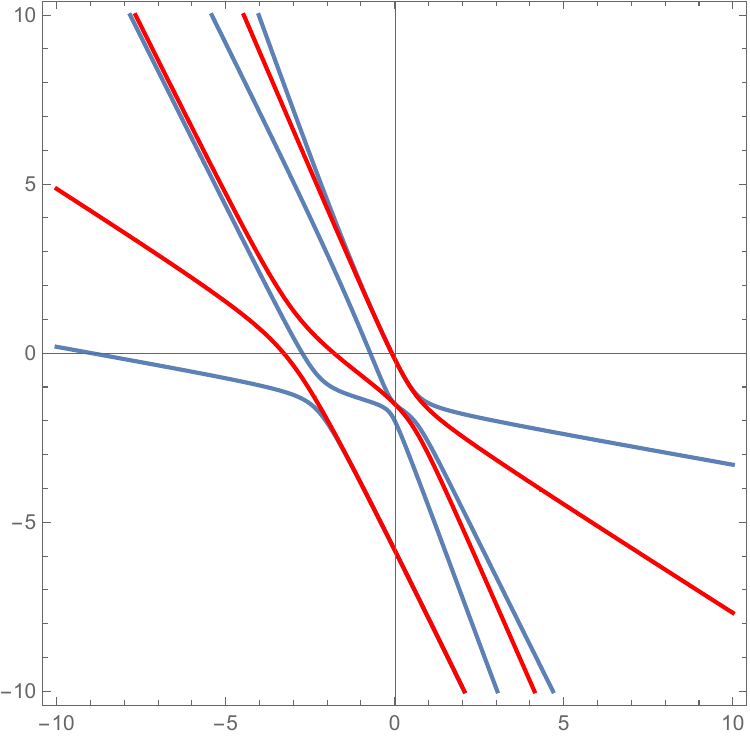}
\caption{\label{interlaced}The Vinnikov curve from a $3$-determinantal measure on $[3]^4$ (in blue) and an interlaced Vinnikov curve from the measure on $[3]^3$ in red.}
\end{center}
\end{figure}

\subsection{Commuting Symmetric $k$-determinantal measures}\label{commutingscn}

If in addition to being symmetric the matrices $A_1,\dots,A_k$ from a $k$-determinantal measure commute with each other,
then the characteristic polynomial $P$ factors into linear factors $P=\prod_{i=1}^k(a_{1i}x_1+a_{2i}x_2+\dots+a_{ki}x_k).$
We can conclude that the measure $\rho$, which counts the number of occurrences of each color, has the distribution of 
a sum of $n$ independent
$k$-sided die rolls, where the $i$th die is biased proportional to the $i$th eigenvalues of the $A_j$.

\section{Determinantal random permutations}\label{permscn}

Let $S_n$ be the permutation group on $[n]=\{1,2,\dots,n\}$.
An \emph{$S_n$-determinantal measure}
is an $n$-determinantal measure on $[n]^n$ in which each $A_i$ has rank $1$.
The rank-$1$ condition implies that each color occurs exactly once, so the measure is supported on 
permutations in $S_n$. 
See below for an example.

Since $A_i$ has rank $1$, we can write $A_i=u_iv_i^t$ where $u_i,v_i\in\R^n$. 
Since $\sum A_i=I$, the $u_i$ are a basis for $\R^n$. 
We have 
$$u_j=\sum_{i=1}^nA_iu_j=\sum_{i=1}^nu_i(v_i\cdot u_j)$$ 
which implies that $v_i\cdot u_j=\delta_{i,j}$.
If we let $U$ be the matrix whose columns are $u_1,\dots,u_n$, and $V$ the matrix whose rows are
$v_1^t,\dots,v_n^t$,
then $VU=I$, and we can conclude that
the measure $\mu$ is defined by the matrix $U$ (or $V$).

By (\ref{singleform}), the probability of a permutation $\sigma\in S_n$ is
\begin{align}
\nonumber\Pr_\mu(\sigma) &= \det((u_{\sigma(1)}v_{\sigma(1)}^1),\dots,(u_{\sigma(n)}v_{\sigma(n)}^n))\\
&=\label{probsig}(\det U)(-1)^\sigma v_{\sigma(1)}^1\dots v_{\sigma(n)}^n.\end{align}

Assume $\det U>0$ (which we can arrange by changing the signs of $u_1$ and $v_1$ if necessary).
The non-negativity of the probabilities puts strong constraints on the matrix $V$. 
If we write the usual expansion of the determinant of $V$:
$$\det V = \sum_{\sigma\in S_n}(-1)^\sigma v_{1\sigma(1)}\dots v_{n\sigma(n)},$$
each term in this expansion is nonnegative, since after scaling by $\det U$ it is equal,  by (\ref{probsig}), to the probability of permutation $\sigma^{-1}$.
By the result of \cite{VY}, $V$ must be a Kasteleyn matrix of a \emph{bipartite Pfaffian graph} (a bipartite graph which admits a Kasteleyn signing).

Conversely, every edge-weighted bipartite Pfaffian graph with $n$ white and $n$ black vertices
determines a determinantal measure on $S_n$: letting $V=K$ the Kasteleyn matrix, and $U=V^{-1}$,
the matrices $A_i$ are given by $A_i=u_iv_i^t$.

\begin{thm} \label{perm} Determinantal permutation measures are exactly those constructed as above 
from bipartite Pfaffian graphs with positive edge weights.
\end{thm}

Although the set of all bipartite Pfaffian graphs does not have a particularly simple description, it includes all bipartite planar graphs.
In \cite{RST} a structural description is given, and an algorithm is presented to determine whether or not
a graph is bipartite Pfaffian. The simplest nonplanar example is given by the Heawood graph, Figure \ref{heawood}.
Its Kasteleyn matrix when all edge weights are $1$ is just the bipartite adjacency matrix:
$$K=\begin{pmatrix}
1 & 0 & 1 & 0 & 0 & 0 & 1 \\
 1 & 1 & 0 & 1 & 0 & 0 & 0 \\
 0 & 1 & 1 & 0 & 1 & 0 & 0 \\
 0 & 0 & 1 & 1 & 0 & 1 & 0 \\
 0 & 0 & 0 & 1 & 1 & 0 & 1 \\
 1 & 0 & 0 & 0 & 1 & 1 & 0 \\
 0 & 1 & 0 & 0 & 0 & 1 & 1
 \end{pmatrix}.$$

The resulting determinantal measure on $S_7$ is uniform on its support, which is the set of the $24$ permutations in $S_7$ 
where for each $i$, 
$\sigma(i)\in\{i+1,i,i-2\}$ modulo $7$. For example $\sigma=1732645$.

\begin{figure}[htbp]
\begin{center}
\includegraphics[width=2in]{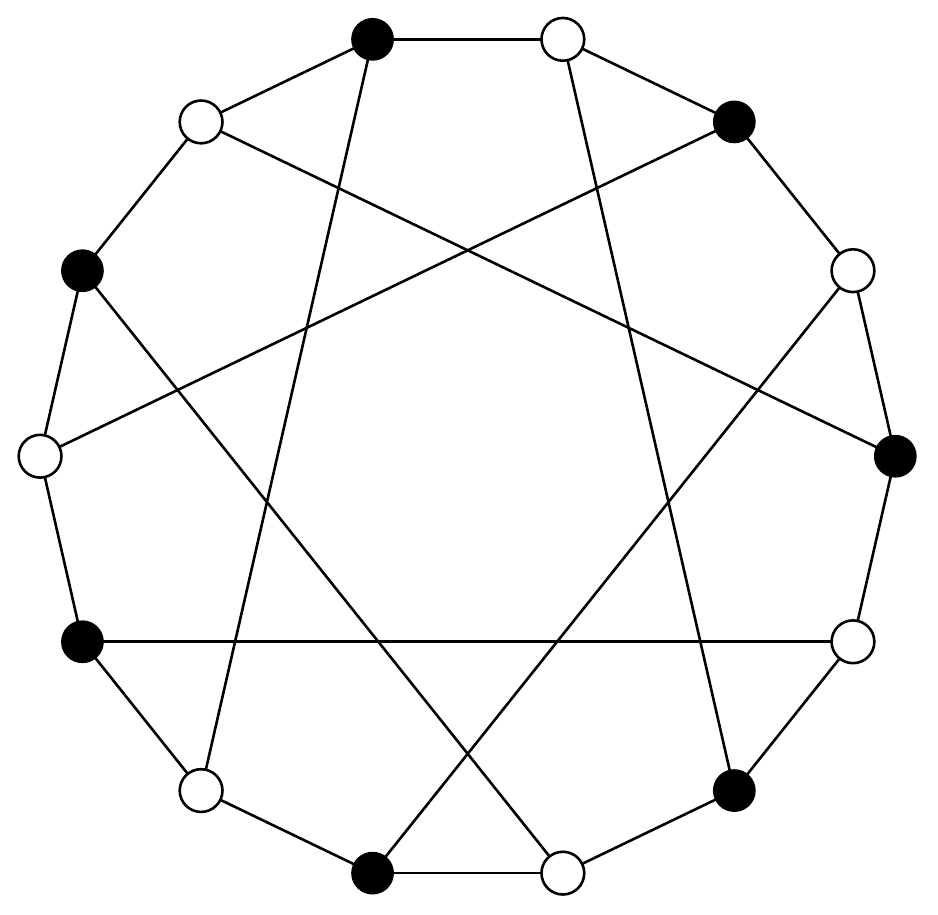}
\caption{\label{heawood}The Heawood graph.}
\end{center}
\end{figure}

\section{Projections and pure $k$-determinantal measures}\label{projscn}

Recall that a $k$-determinantal measure $\mu$ is \emph{pure} if the ranks of the $A_i$ add up to $n$.
In this case, by Lemma \ref{pureproj}, each $A_i$ is a projection matrix. Let $V_i$ be the subspace which is the image of $A_i$.
As shown in that lemma, $A_i$ is the projection to $V_i$ along the span of the other $V_j$'s. 

Let $n_i$ be the rank of $A_i$. 
Let $M$ be an $n\times n$ matrix whose first $n_1$ columns span $V_1$, next $n_2$ columns span $V_2$, and so on.
This matrix $M$ is well-defined up to the right action by $\GL_{n_1}(\R)\times\dots\times\GL_{n_k}(\R)$, which performs column operations on the first $n_1$
columns of $M$, the next $n_2$ columns of $M$, and so on. Note that $M$ has full rank.

We have 
\be\label{Adef}A_i = M\begin{pmatrix}0&0&0\\0&I_{n_i}&0\\0&0&0\end{pmatrix}M^{-1},\ee
where the central matrix has the $n_i\times n_i$ identity matrix $I_{n_i}$ occurring starting at the appropriate index $n_1+\dots+n_{i-1}+1$. 

Let $L=M^{-1}$. 
We partition $[n]$ into subsets $U_1,\dots U_k$, with
$$U_1=\{1,2,\dots,n_1\},~U_2=\{n_1+1,\dots,n_1+n_2\},\dots, U_k=\{n-n_k+1,\dots,n\}.$$
A point $\pi\in[k]^n$ in the support of $\mu$ is a permutation of $1^{n_1}2^{n_2}\dots k^{n_k}$, the sequence of $n_1$ `$1$'s, 
$n_2$ `$2$'s, etc. For each $j$ let $\pi(U_j)\subset[n]$ denote the locations of the indices $j$ in $\pi$.

\begin{thm}\label{minors} Single point probabilities for $\mu$ are products of minors of $L$, up to a multiplicative constant: we have
\be\label{prod}\Pr(\pi) = (-1)^\pi Q\prod_{i=1}^k L_{U_i}^{\pi(U_i)}\ee
where $Q=\det M$ and $(-1)^\pi$ is the signature of the mapping $\pi$, thought of as a permutation from $[n]$ to $[n]$. 
\end{thm}

For example when $k=3$, and $(n_1,n_2,n_3)=(2,2,2)$ we have the point probability
\be\label{ex1}\Pr((2,1,1,3,2,3)) = -QL_{12}^{23}L_{34}^{15}L_{56}^{46}.\ee
Here $\pi:112233\to211323$ has corresponding permutation $123456\to231546$ and signature $-1$.

\begin{proof} The general pattern can be seen by working out a sufficiently general example such as (\ref{ex1}).
Using (\ref{singleform}), we have $\Pr((2,1,1,3,2,3))=$
\begin{align*}=&\det(A_2^1,A_1^2,A_1^3,A_3^4,A_2^5,A_3^6)\\
=&\det\left(\left[M\!\!\begin{pmatrix}
0 & 0 & 0 & 0 & 0 & 0 \\
 0 & 0 & 0 & 0 & 0 & 0 \\
 0 & 0 & 1 & 0 & 0 & 0 \\
 0 & 0 & 0 & 1 & 0 & 0 \\
 0 & 0 & 0 & 0 & 0 & 0 \\
 0 & 0 & 0 & 0 & 0 & 0 
 \end{pmatrix}L\right]^1,
 \dots,
 \left[M\!\!\begin{pmatrix}
0 & 0 & 0 & 0 & 0 & 0 \\
 0 & 0 & 0 & 0 & 0 & 0 \\
 0 & 0 & 0 & 0 & 0 & 0 \\
 0 & 0 & 0 & 0 & 0 & 0 \\
 0 & 0 & 0 & 0 & 1 & 0 \\
 0 & 0 & 0 & 0 & 0 &1 
 \end{pmatrix}\!\!L\right]^6\right).
\end{align*}
 We can factor the $M$'s out from the left.
What remains is
$$\det M\det\begin{pmatrix}
0&L_{12}&L_{13}&0&0&0\\
0&L_{22}&L_{23}&0&0&0\\
L_{31}&0&0&0&L_{35}&0\\
L_{41}&0&0&0&L_{45}&0\\
0&0&0&L_{54}&0&L_{56}\\
0&0&0&L_{64}&0&L_{66}\\
\end{pmatrix},$$
which can be evaluated in block form as $\pm Q L_{12}^{23}L_{34}^{15}L_{56}^{46}$.
The sign is the sign of the permutation of $112233\to211323$, acting on columns (in this case the sign is $-1$).
\end{proof}

The matrix $L$ is only defined up to the left action of $\GL_{n_1}(\R)\times\dots\times\GL_{n_k}(\R)$,
performing row operations on rows in each $U_i$. It is thus convenient to think of $L$ as a $k$-tuple of 
Grassmannian elements $L_i\in Gr_{n_i,n}$ where $L_i$ consists of the rows $U_i$ of $L$.

Theorem \ref{minors} then has the following corollary.
\begin{corollary}
Every pure $k$-determinantal process with $A_i$ of rank $n_i$ can be constructed from elements $L_i=L_{U_i}\in Gr_{n_i,n}$ (for $1\le i\le k$)
with the property that all products of Pl\"ucker coordinates (\ref{prod}) are nonnegative. 
\end{corollary}

A simpler statement for $k=2$ is given below in Section \ref{mnevsection}.

\subsection{Supports of pure $k$-determinantal measures, $k\ge3$}

Theorem \ref{minors} implies that, for $k\ge 3$, the supports of pure $k$-determinantal measures have certain restrictions:
they cannot be supported on all $\binom{n}{n_1,\dots,n_k}$ sequences.
Take any three distinct indices $i_1,i_2,i_3\in[n].$ Let $x\in[k]^n$ have $3$ distinct color values at $i_1,i_2,i_3$, for example
$x_{i_1}=1,x_{i_2}=2,x_{i_3}=3$. Consider the other two points $x',x''$ obtained by cylically permuting the
values at these three indices, that is, $x',x''=x$ at all indices except $i_1,i_2,i_3$ and $x'_{i_1}=2,x'_{i_2}=3,x'_{i_3}=1$,
and $x''_{i_1}=3,x''_{i_2}=1,x''_{i_3}=2$.
Then we claim that $\Pr(x),\Pr(x'),\Pr(x'')$ cannot all be positive. 
To see this, let $y,y',y''$ be the other three points agreeing with $x$ off of $i_1,i_2,i_3$ (and thus having the other three permutations of $1,2,3$ at $i_1,i_2,i_3$).
By (\ref{prod}), 
$$\Pr(x)\Pr(x')\Pr(x'')=-\Pr(y)\Pr(y')\Pr(y'')$$ (both sides have the same $L$ minors but differing signs), a contradiction to the positivity of either side, so both sides must be zero.

Can this fact be used to give a different proof of Theorem \ref{perm}?

\subsection{Characterization of pure $2$-determinantal measures}\label{mnevsection}

Contrary to the previous section, one can construct pure $2$-determinantal measures with full support. 
We give a construction arising from a pair of matrices in the Grassmannian, with Pl\"ucker coordinates of the same sign.
This construction is general in the sense that it gives all pure $2$-determinantal measures.

Let $n_1+n_2=n$, and
let $G_1,G_2\in Gr_{n_1,n}$ be two elements whose corresponding Pl\"ucker coordinates have the same signs.
(We assume we are in the generic situation where all Pl\"ucker coordinates are nonzero; the nongeneric situation can be obtained
from a limit of generic cases.)
Choose representative $n_1\times n$ matrices $(I_{n_1}~A)$ for $G_1$ and $(I_{n_1}~B)$ for $G_2$.
Let
\be\label{L}L=\begin{pmatrix}I_{n_1}&A\\-B^t&I_{n_2}\end{pmatrix}.\ee 

Then we claim that $A_1,A_2$ defined by (\ref{Adef}) (with $L=M^{-1}$) form a $2$-determinantal measure.
To show this, from Theorem \ref{minors} we need to show that for all $\pi\in [2]^n$
\be\label{pp}\Pr(\pi) = (-1)^\pi QL_{U_1}^{\pi(U_1)}L_{U_2}^{\pi(U_2)}>0,\ee
where $L_{U_1}=(I_{n_1}~A)$ consists of the first $n_1$ rows of $L$ and $L_{U_2}=(-B^t~I_{n_2})$ consists of the last $n_2$ rows. 

We have
$$Q=\det M =(\det L)^{-1} = (\det(I+AB^t))^{-1}>0$$ by Lemma \ref{detpos} below.
It suffices then to prove that the sign of
$L_{U_1}^{\pi(U_1)}L_{U_2}^{\pi(U_2)}$ is $(-1)^\pi$. 

The signature $(-1)^\pi$ can be computed as the parity of the total displacement of $1$s, or equivalently 
the total displacement of $2$s: 
$(-1)^\pi=(-1)^s$ where $s=\sum_{i\in U_1}\pi(i)-i=\sum_{j\in U_2}j-\pi(j)$, since each term in the sum
is the number of $12\to21$ crossings each index makes. 

By Lemma \ref{compl} below, maximal minors of $L_{U_2}=(-B^t~I_{n_2})$ are, up to sign, equal to the complementary maximal minors of $S:=(I_{n_1}~B)$, with sign change given by the sum of ``index displacements'':
that is, for a subset $J\subset[n]$ of size $n_2$ we have 
$$(L_{U_2})^J=(-1)^JS^{J^c}$$
where $(-1)^J=(-1)^{(j_1-1)+\dots+(j_{n_2}-n_2)}.$ If $J=\pi(U_2)$ this sign is exactly $(-1)^\pi$.

By hypothesis, the sign of $S^{J^c}$ equals that of $(L_{U_1})^{J^c}$, so setting
$J=\pi(U_2)$, the signs cancel in (\ref{pp}).
This shows that (\ref{pp}) is positive, and completes the construction  of a pure $2$-determinantal measure.
\bigskip

It is unclear how to explicitly parameterize all such pairs $G_1,G_2$ (and thereby parameterize pure $2$-determinant processes). 
In fact if we include $0$ as a potential sign (that is, restrict some minors to be zero), 
the Mn\"ev Universality Theorem \cite{mnev} says that subsets of 
$\Gr_{n_1,n}$ having Pl\"ucker coordinates of predetermined signs in $\{-1,0,1\}$ can be arbitrarily complicated,
essentially equivalent to any semialgebraic set. So in this sense there seems to be little hope for a reasonable parameterization of such pairs $G_1,G_2$. 

\section{Appendix: Linear Algebra}

\begin{lemma} \label{detpos}
If $n_1\times n$ matrices $(I_{n_1}~A)$ and $(I_{n_1}~B)$ have corresponding Pl\"ucker coordinates of the same signs (with some pair nonzero) then $\det(I+AB^t)>0$.
\end{lemma}

\begin{proof} We have
$$\det(I+AB^t) = \sum_{S\subset[n]}\det((AB^t)_S^S) = \sum_{S\subset[n]}\sum_{\substack{T\subset[n_1]\\|T|=|S|}}A_S^T(B^t)_T^S$$
where the last equality is the Cauchy-Binet Theorem. Each minor $A_S^T$ is a Pl\"ucker coordinate of $(I_{n_1}~A)$ and
$(B^t)_T^S$ is the corresponding Pl\"ucker coordinate of $(I_{n_1}~B)$. These have the same sign by hypothesis 
so every term in the double sum is nonnegative.
\end{proof}

\begin{lemma}\label{compl} With $n=n_1+n_2$ let $S=(I_{n_1}~B)\in Gr_{n_1,n}$ and $R=(-B^t~I_{n_2})\in Gr_{n_2,n}$. 
For a set $J\subset[n]$ of size $n_1$ we have the identity of Pl\"ucker coordinates $S^J=(-1)^JR^{J^c}$
where $(-1)^J=(-1)^{(j_1-1)+\dots+(j_{n_1}-n_1)}.$ 
\end{lemma}

\begin{proof} Let $J=J_1\cup J_2$ where $J_1=J\cap[n_1]$ be the part of $J$ within the first $n_1$ indices, and 
$J_2=J\cap\{n_1+1,\dots,n\}$. Let $\bar{J}_1=[n_1]\setminus J_1$ and $\bar{J}_2=\{n_1+1,\dots,n\}\setminus J_2$.
Then $S^J=(-1)^XB_{\bar{J}_1}^{J_2}$ where $(-1)^X$ is the sign of the permutation of rows moving rows $J_1$ to the
first $|J_1|$ rows.
Also $R^{J^c}= (-1)^Y(-B^t)_{J_2}^{\bar{J}_1}= (-1)^{Y+|J_2|}B_{\bar{J}_1}^{J_2}$
where $(-1)^Y$ represents the signs of the permutation switching rows of $R$ to make rows corresponding to $\bar J_2$ 
 the last $|\bar J_2|$ rows, that is, so that $R^{J^c}$ is transformed into block form 
 $$\begin{vmatrix}(-B^t)_{J_2}^{\bar{J}_1}&0\\C&I_{|\bar J_2|}\end{vmatrix}.$$
 
To check that the signs are correct as stated, note that changing one element of $J$ to a consecutive index changes the sign,
in both $(-1)^J$ and the ratio of signs of the two minors. 
So we just need to check signs for a single choice of $J$. An easy case is when $J=[n_1]$, for which $S^J=1=R^{J^c}$and $(-1)^J=1$.
\end{proof}

\bibliographystyle{plain}
\bibliography{rgb}
\end{document}